\documentclass[11pt]{article}
\usepackage{setspace}
\usepackage{amsfonts}
\usepackage{amsmath}

\usepackage{hyperref}
\usepackage{amsthm}
\usepackage{amssymb}
\usepackage[top=1in, bottom=1in, left=1.2in, right=1.2in]{geometry}
\usepackage{color}

\newtheorem{theorem}{Theorem}[section]
\newtheorem{problem}[theorem]{Problem}
\newtheorem{prop}[theorem]{Proposition}

\newtheorem{lemma}[theorem]{Lemma}
\newtheorem{observation}[theorem]{Observation}
\newtheorem{claim}[theorem]{Claim}

\newtheorem{conjecture}[theorem]{Conjecture}

\begin{document}
\title{A stability theorem for maximal $K_{r+1}$-free graphs}
\date{\vspace{-5ex}}

\author{Kamil Popielarz \thanks{Department of Mathematics, University of Memphis, Memphis, Tennessee; \textit{kamil.popielarz@gmail.com}}
       \and
       Julian Sahasrabudhe \thanks{Department of Mathematics, University of Memphis, Memphis, Tennessee; \textit{julian.sahasra@gmail.com}}
       \and
       Richard Snyder\thanks{Department of Mathematics, University of Memphis, Memphis, Tennessee; \textit{rsnyder1@memphis.edu}}
      }
    
\maketitle

\begin{abstract}
For $r \geq 2$, we show that every maximal $K_{r+1}$-free graph $G$ on $n$ vertices with \\$(1-\frac{1}{r})\frac{n^2}{2}-o(n^{\frac{r+1}{r}})$ edges contains a complete $r$-partite subgraph on $(1 - o(1))n$ vertices. We also show that this is best possible. This result answers a question of Tyomkyn and Uzzell.        
\end{abstract}

\onehalfspace
\section{Introduction}

For a positive integer $r \geq 2$, a graph $G$ is said to be $(r+1)$\textit{-saturated} (or \textit{maximal $K_{r+1}$-free}) if it contains no copy of $K_{r+1}$, but the addition of any edge from the complement $\overline{G}$ creates at least one copy of $K_{r+1}$. Let $T_{r}(n)$ denote the $r$-\emph{partite Tur\'{a}n graph} that is, the $n$-vertex, complete $r$-partite graph for which each of the $r$ classes is of order $\lfloor n/r \rfloor $ or $\lceil n/r \rceil$. We write $t_{r}(n) = e(T_{r}(n))$, and note that $t_r(n) = (1-\frac{1}{r})\frac{n^2}{2} + O_r(1)$. Whenever we speak of an $r$-partite subgraph, we require that it is induced.

The classical theorem of Tur\'{a}n \cite{t} tells us that, for an integer $r \geq 2$, the maximum number of edges in a graph not containing a $K_{r+1}$ is $t_r(n)$, and that $T_r(n)$ is the unique $K_{r+1}$-free graph attaining this maximum. Erd\H{o}s and Simonovits \cite{erd1, erd2, s} discovered that this extremal problem exhibits a certain `stability' phenomenon: $K_{r+1}$-free graphs for which $e(G)$ is close to $t_r(n)$ must resemble the Tur\'{a}n graph in an appropriate sense. In particular, they proved that every $n$-vertex, $K_{r+1}$-free graph with at least $t_r(n) - o(n^2)$ edges can be transformed into $T_r(n)$ by making at most $o(n^2)$ edge deletions and additions. 

Beyond the seminal work of Erd\H{o}s and Simonovits, we are lead to consider finer aspects of this phenomenon. More generally, it is natural to ask how the structure of a $K_{r+1}$-free graph $G$ comes to resemble the Tur\'{a}n graph as the number of edges $e(G)$ approaches the Tur\'{a}n number $t_r(n)$. For instance, Nikiforov and Rousseau~\cite{Nikiforov_Rousseau}, in the context of a Ramsey-theoretic problem, showed that for $r \ge 2$ and $\varepsilon$ sufficiently small (depending on $r$) the following holds: if $G$ is an $n$-vertex $K_{r+1}$-free graph with $e(G) \ge \left(1 - \frac{1}{r} - \varepsilon\right)n^2/2$, then $G$ contains an induced $r$-partite subgraph $H$ with $|H| \ge (1 - 2\varepsilon^{1/3})n$ and $\delta(H) \ge \left(1 - \frac{1}{r} - 4\varepsilon^{1/3}\right)n$. In other words, $G$ must contain a large $r$-partite subgraph with minimum degree almost as large as $\delta(T_r(n))$. The interested reader should consult the survey of Nikiforov~\cite{Nikiforov} for a few other stability results in a similar vein.

Another result concerning the finer structure of stability is due to Brouwer~\cite{b}, who showed that if $n \geq 2r+1$ and $G$ is a $K_{r+1}$-free graph with $e(G) \geq t_r(n) - \lfloor \frac{n}{r} \rfloor +2$, then $G$ must be $r$-partite. This result has further been rediscovered by several authors~\cite{afgs, ht, kp}, and Tyomkyn and Uzzell~\cite{tu} recently gave a new proof. In this paper, we are interested in the structure of \emph{maximal} $K_{r+1}$-free graphs near the Tur\'{a}n threshold. In this context, Brouwer's result says that if the number of edges of an $(r+1)$-saturated graph $G$ is roughly within $n/r$ of the Tur\'{a}n number $t_r(n)$, then $G$ \emph{is} complete $r$-partite. A natural question then arises, which informally is: when can one guarantee `almost-spanning' complete $r$-partite subgraphs in $(r+1)$-saturated graphs?

Continuing this line of investigation, Tyomkyn and Uzzell \cite{tu} proved, among other results, that every $4$-saturated graph on $n$ vertices and with $t_3(n)- cn$ edges contains a complete $3$-partite graph on $(1-o(1))n$ vertices (they also implicitly dealt 
with the $3$-saturated case). They went on to ask if one can similarly find almost-spanning, complete $r$-partite subgraphs in $(r+1)$-saturated graphs with many edges, for $r \geq 4$. The main result of this paper is to resolve the question of Tyomkyn and Uzzell, in a stronger form. Not only do we show that this phenomenon persists for $(r+1)$-saturated graphs for all $r \geq 2$, but we also determine the edge threshold for which the result fails to hold. In particular, we show the following.

\begin{theorem} \label{thm:MainTheoremSoftForm}
Let $r \geq 2$ be an integer. Every $(r+1)$-saturated graph $G$ on $n$ vertices with $t_r(n) - o(n^{\frac{r+1}{r}})$ edges contains a complete $r$-partite subgraph on $(1-o(1))n$ vertices.  
\end{theorem}
We also show that this theorem is tight in the sense that for every $\delta >0$ there exist graphs $G$ with $t_r(n)-\delta n^{\frac{r+1}{r}}$ edges 
for which the conclusion of Theorem~\ref{thm:MainTheoremSoftForm} fails. 

We actually deduce Theorem \ref{thm:MainTheoremSoftForm} from a stronger, quantitative result, which we now make precise. For a graph $G$ and an integer $r \geq 2$, define the graph parameter 
\[ g_r(G) = \min\{ |T| : T \subseteq V(G), G-T \text{ is complete $r$-partite}\}.\] For $n,m \in \mathbb{N}$, let $\mathcal{S}_r(n,m)$ denote the set of all $(r+1)$-saturated graphs on $n$ vertices with at least $t_r(n) - m$ edges. Then define 
\[
	g_r(n,m) = \max\{ g_r(G): G \in \mathcal{S}_r(n,m)\}. 
\]
The quantitative form of our main theorem, stated below, gives an upper bound for the function $g_r(n,m)$ under some modest conditions on $n$.

\begin{theorem}\label{thm2}
Let $r,n$ be integers satisfying $r \geq 2$, $n \geq 900r^6$. Every $(r+1)$-saturated graph with $t_r(n)- m$ edges contains a complete $r$-partite subgraph on $(1-C_rmn^{-\frac{r+1}{r}})n$ vertices, where $C_r$ is a constant depending only on $r$.
\end{theorem}
We shall also give a construction in Section~\ref{section_constructions} showing that this result is tight, up to the value of $C_r$, in a certain range of $m$.
More precisely, if $\varepsilon > 0$, $n \ge 2^{10r}/\varepsilon$ and $(\frac{r-1}{r}+\varepsilon)n \le m \le n^{\frac{r+1}{r}}$, then
\[
    c_{r, \varepsilon}mn^{-1/r} \leq  g_r(n, m) \leq C_rmn^{-1/r}, 
\] 
where $c_{r, \varepsilon}$ is a constant depending on $r$ and $\varepsilon$, and $C_r$ is a constant depending only on $r$.
This explicit form of our main result takes a major step towards a further question of Tyomkyn and Uzzell \cite{tu, TU_euro}, who asked for the determination of $g_3(n,cn)$. 
While we have determined $g_r(n, m)$ up to constants for $m \in \left[ (\frac{r-1}{r}+\varepsilon)n, n^{\frac{r+1}{r}} \right]$, our construction giving the lower bound does not work for $m \in \left[ \frac{n}{r}, \frac{r-1}{r}n \right]$.
We leave the determination of $g_r(n, m)$ in this range as an open problem (see Section~\ref{final_remarks}).

We also consider the situation for $(r+1)$-saturated graphs with $t_r(n) - Cn^{\frac{r+1}{r}}$ edges; that is, just beyond the edge threshold in Theorem~\ref{thm:MainTheoremSoftForm}. In this range it is perhaps most natural to consider ``balanced'' $r$-partite complete subgraphs or, in other words, $r$-partite \emph{Tur\'{a}n} subgraphs. With this in mind we set 
\[
	g^*_r(G) = \min\{|T|: G-T \text{ is an $r$-partite Tur\'{a}n graph}\},
\]
and, for $m,n\in \mathbb{N}$, define \[
	g^*_r(n, m) = \max\{ g^*_r(G): G \in \mathcal{S}_r(n, m)\}. 
\]
Thus, $g^*_r(n, m)$ is the maximum number of vertices one is required to delete from an $(r+1)$-saturated graph on $n$ vertices with at least $t_r(n) - m$ edges such that the remaining graph is an $r$-partite Tur\'{a}n graph. While it is not hard to see that the functions $g_r^*(n,m)$ and $g_r(n,m)$ are very closely related in the range $m = o(n^{\frac{r+1}{r}})$ (as $n \rightarrow \infty$), they take on a somewhat different behaviour when $m \geq Cn^{\frac{r+1}{r}}$. In this range, $g_r^*$ becomes the more natural parameter of study. We show that $g_r^*(n,Cn^{\frac{r+1}{r}})$ increases rapidly as $C$ increases. 


\begin{theorem}\label{thm:RandomConstruction}
Let $r \geq 2 $ be an integer and let $\delta > 0$. There exists a constant $C = C(r,\delta)$ such that, for $n$ sufficiently large, there exists an $n$-vertex $(r+1)$-saturated graph $G$ that contains no copy of $T_r(\delta r n)$ and $e(G) \geq t_r(n) - Cn^{\frac{r+1}{r}}$. In other words, for any sufficiently large $D >0$ we have 
\[g_r^*(n,Dn^{\frac{r+1}{r}}) \geq \left(1-\frac{c'\log(Dr)}{D} \right)n ,
\] for sufficiently large $n$ and an absolute constant $c'$.

\end{theorem}
We do not have any corresponding upper bounds on $g_r^*(n,m)$ in this range of $m$.
 \subsection{Organization and Notation}
The rest of the paper is organized as follows. In Section~\ref{sec2}, we prove our main result, Theorem~\ref{thm2}. Roughly speaking, we first show that any $K_{r+1}$-free graph with many edges has a rather substantial $r$-partite subgraph. We then show that one can refine this resultant $r$-partite graph by making each bipartite graph between partition classes complete, while removing relatively few vertices. In Section~\ref{section_constructions}, we provide the aforementioned constructions which exhibit the tightness of Theorems~\ref{thm:MainTheoremSoftForm} and~\ref{thm2}; in Section~\ref{balanced}, we prove Theorem~\ref{thm:RandomConstruction}. Finally, in Section~\ref{final_remarks} we state some further questions. 

Our notation is mostly standard (see, for example, \cite{bol}). For a subset $S \subseteq V(G)$ we denote by $N_G(S) = \bigcap_{v\in S} N_G(v)$ the common (or joint) neighbourhood of $S$ in $G$. We shall omit the subscript `$G$' if the underlying graph is understood. If $X_1, \dots, X_r$ are disjoint subsets of $V(G)$, we denote by $G[X_1, \dots, X_r]$ the $r$-partite graph induced in $G$ with vertex classes $X_1, \dots, X_r$. We write 
$f \ll g$ to mean $f(n)/g(n) \rightarrow 0$ as $n\rightarrow \infty$.  All other notation we need shall be introduced as necessary.


\section{The Proof of Theorem~\ref{thm2}}\label{sec2}

\subsection{Preliminary lemmas}\label{subsection_PrelimLemmas}
Let us now work towards establishing Theorem~\ref{thm2}. For that we state and prove two lemmas, the second of which is the core of the proof. For the first lemma we use the following theorem of Andr{\'a}sfai, Erd\H{o}s, and S{\'o}s \cite{aes}, although the precise value of the constant $\frac{3r-4}{3r-1}$ is unimportant for us; we only need that it is strictly less than the Tur{\'a}n density.

\begin{theorem}\label{thm3}
For $r\geq 2$ let $G$ be a $K_{r+1}$-free graph on $n$ vertices which is not $r$-partite. Then there is a vertex $v$ of $G$ with
\[ d(v) \leq \frac{3r-4}{3r-1}n. \]
\end{theorem}

We shall also use the following result of Brouwer \cite{b}, mentioned in the introduction. 

\begin{theorem} \label{thm:KrfreeVeryClose}
Let $r \geq 2$, $n \geq 2r+1$, and let $G$ be an $K_{r+1}$-free, $n$-vertex graph. If $e(G) \geq t_r(n) - \lfloor \frac{n}{r} \rfloor +2$, then $G$ is $r$-partite.
\end{theorem}
Here, then, is our first lemma, which grants us a sizable induced $r$-partite subgraph. We remark that a lemma of this type is not new and appears in a similar form in~\cite{tu}.

\begin{lemma}\label{lemma1}
For $r\geq 2$ there is a constant $d_r$, depending only on $r$, such that the following holds. Let $n \geq 4r$ and $0\leq \varepsilon \leq (30r^3)^{-1}$. If $G$ is an $n$-vertex $K_{r+1}$-free graph with $e(G) \geq  t_r(n) - \varepsilon n^2$, then there is a subset $T\subseteq V(G)$ with $|T| \leq d_r\varepsilon n$ such that $G-T$ is $r$-partite.
\end{lemma}
\begin{proof}
If $\varepsilon < (2rn)^{-1}$,  then $e(G) > t_r(n) - \frac{n}{2r} \geq t_r(n)- \lfloor \frac{n}{r} \rfloor +1$, where the second inequality follows by our assumption that $n \geq 4r$. Therefore by Theorem \ref{thm:KrfreeVeryClose}, $G$ is $r$-partite, and there is nothing to prove. Accordingly, we may assume $\varepsilon \geq (2rn)^{-1}$.

Set $G_1 = G$. Suppose that $G_1, \dots, G_i$ have been defined for some $i \in [n]$. If $G_i$ is not $r$-partite then pick a vertex $v_i \in V(G_i)$ with $d_{G_i}(v_i) \leq \frac{3r-4}{3r-1}|G_i|$ according to Theorem~\ref{thm3}. Set $G_{i+1} = G_i - v_i$. Suppose this process terminates at stage $t\in [n]$. Then $G_{t+1} = G - \{v_1, \dots, v_t \}$ is $r$-partite. We claim that $t \leq d_r\varepsilon n$ for some constant $d_r$ depending only on $r$. This follows from a simple calculation. Indeed as $e(G_{i+1}) \leq \frac{r-1}{2r}(n-i)^2$ holds for every $i \in [t]$, by Tur{\'a}n's theorem we have
\begin{align*}
e(G) &\leq \frac{3r-4}{3r-1}\bigg(n+ (n-1) + \dots + (n-i+1)\bigg) + \frac{r-1}{2r}(n-i)^2 \\
&= \frac{3r-4}{3r-1}\left(ni - {i\choose 2}\right) + \frac{r-1}{2r}(n-i)^2,
\end{align*}
and using the lower bound on $e(G)$ we obtain
\begin{equation}
t_r(n) - \frac{r-1}{2r}(n-i)^2 + \frac{3r-4}{3r-1}{i\choose2}\,\leq \, \frac{3r-4}{3r-1}ni + \varepsilon n^2. \label{eq:turan1}
\end{equation}
Further, using the lower bound $t_r(n) \geq (1-1/r){n\choose 2}$ applied to (\ref{eq:turan1}) and rearranging yields the equivalent inequality

\begin{equation*}
i\left(1-\frac{i}{2n}-\frac{r(3r-4)}{2n}\right) - \frac{1}{2}(r-1)(3r-1) \,\leq \, r(3r-1)\varepsilon n ,
\end{equation*}
which is easily shown to fail if $i = 10r^2(3r-1)\varepsilon n$ when $(2rn)^{-1}\leq \varepsilon \leq (30r^3)^{-1}$. Since the resulting function in (\ref{eq:turan1}) is quadratic in $i$, it is indeed enough to demonstrate that it fails for one value. Accordingly, $t < 10r^2(3r-1)\varepsilon n$ as claimed.
\end{proof}

The next lemma is the heart of the proof of our main theorem. Before stating it we introduce some notation and a bit of terminology. If $G$ is an $r$-partite graph with vertex partition $V_1,\ldots,V_r$, then we denote by $\widetilde{G}[V_1, \dots, V_r]$ the $r$-\emph{partite complement of} $G$ with respect to the partition $V_1,\ldots,V_r$. In other words $\widetilde{G}[V_1, \dots , V_r]$ has vertex set $V_1\cup \dots \cup V_r$ and its edges are precisely the non-edges of $G$ which join two vertices belonging to distinct vertex classes of $V_1,\ldots,V_r$. Often we simply speak of \emph{the} $r$-\emph{partite complement} in the case that the vertex partition we are using is clear from context, and we shall simply write $\widetilde{G}$. We say that a subset $S \subseteq V(G)$ of the vertices of a graph $G$ \textit{covers} an edge $e$ if at least one of the endpoints of $e$ lies in $S$. Further, we let $I_G(S)$ denote the collection of edges of $G$ covered by $S$. An $r$-\textit{saturating edge} in $G$ is an edge of the complement $\overline{G}$ the addition of which creates a copy of $K_r$ in $G$. If $X, Y \subseteq V(G)$ are subsets of vertices, then we say that a non-edge $e$ is an $r$-\emph{saturating} $(X,Y)$ \emph{edge} if it is $r$-saturating with one endpoint in $X$ and the other in $Y$.  A $K_r$-\textit{matching} in a graph $G$ is a collection of vertex disjoint copies of $K_r$ in $G$. Lastly, before stating and proving the lemma, let us collect a simple observation that will be of use.
\begin{observation}\label{obs1}
Suppose that $G$ is a bipartite graph with vertex classes $V_1$ and $V_2$ with $e(G) = \alpha|V_1||V_2|$, where $\alpha \in [0,1]$. Then for any $1 \leq t \leq |V_2|$ there is a subset $W \subseteq V_2$ of size $t$ such that the induced graph on $V_1 \cup W$ has at least $\alpha|V_1|t$ edges.
\end{observation}
\begin{proof}
This assertion follows from a simple averaging argument. For $Y \subseteq V_2$ let $e(V_1, Y)$ denote the number of edges of $G$ with an endpoint in $Y$. Then \[ \sum_{Y \in V_2^{(t)}} e(V_1, Y)= e(G)\binom{|V_2|-1}{t-1} = \alpha|V_1|t\binom{|V_2|}{t}, \]
so there exists a subset $W \in V_2^{(t)}$ with $e(V_1, W) \geq \alpha|V_1|t$.
\end{proof}

\begin{lemma}\label{lemma2}
Let $r\geq 2$ be an integer and let $G$ be a $K_r$-free, $r$-partite graph with vertex classes $A$, $B,X_1 \dots , X_{r-2}$. Then the following statements hold.
\begin{enumerate}
\item \label{lemma2:part1} There is a subset $R \subseteq A \cup B$ that covers all $r$-saturating $(A,B)$ edges in $G$ and
\[ |I_{\widetilde{G}}(R)|  \geq  c_r|R|^{\frac{r}{r-1}}, \]
for some constant $c_r >0$ depending only on $r$.
\item \label{lemma2:part2} Suppose that $t \geq 1 $ is an integer with $r-t \geq 2$, that $E \subseteq E_{\widetilde{G}}(A, B)$ is a collection of non-edges between $A,B$, and that there exist $K_{r-t}$-free subgraphs $H_1,\ldots,H_s \subseteq G$ such that every element of $E$ is $(r-t)$-saturating in at least one of the graphs $H_1,\ldots,H_s$. Then there exists a set $R' \subseteq A \cup B$ covering every element of $E$ with
\[ |I_{\widetilde{G}}(R')|  \geq c'_{r,t}s^{-\frac{1}{r-t-1}}|R'|^{\frac{r-t}{r-t-1}}, \]
where $c'_{r,t}$ is a constant depending only on $r,t$.
\end{enumerate}
\end{lemma}
\begin{proof} 
We prove these two statements simultaneously by induction on $r$. The case $r=2$ is trivial: $G$ must be empty. The first part holds by simply choosing the smaller of the two parts of the bipartite graph $G$ and the second part of the statement is vacuous as there is no appropriate choice for $t$.

So, assuming that the result holds for $r-1 \geq 2$, we prove it for $r$. To this end, let $G$ be a $K_r$-free, $r$-partite graph with vertex sets $A$, $B$, $X_1, \dots , X_{r-2}$. We start with the proof of Part \ref{lemma2:part2} as we shall need it to prove Part \ref{lemma2:part1}.

\emph{Proof of Part~\ref{lemma2:part2}:} Suppose we are given a collection $E$ of non-edges between $A,B$ and subgraphs $H_1,\ldots,H_s$ satisfying the requirements of the lemma. Start by enumerating the collection of subgraphs 
\[ \left\lbrace H_i\left[A \cup B \cup X_{i_1} \cup \cdots \cup X_{i_{r-t-2}}\right] : i \in [s], 1\leq i_1 < \dots < i_{r-t-2} \leq r-2  \right\rbrace
\] by $H'_1,\ldots,H'_{s'}$, where $s' = \binom{r-2}{r-t-2}s$ (if $t = r-2$, then we are just listing the subgraphs $H_i[A \cup B]$ for $i =1, \ldots, s$). We now iteratively apply induction inside each of the graphs $H'_1,\ldots ,H'_{s'}$: at each stage we remove a set granted by the induction hypothesis before moving to the next graph in the enumeration.

We shall define a sequence of disjoint subsets $R_1, \dots , R_{s'}$ of $A\cup B$ and a sequence of subgraphs $G_1, \ldots, G_{s'+1}$ of $G$ with the following properties:
\begin{enumerate}
\item $G_1 = G$ and $G_{i+1} = G_i-R_i$ for all $i\geq 1$.
\item $|I_{\widetilde{G}_{i}}(R_i)| \geq c_{r,t}|R_i|^{\frac{r-t}{r-t-1}}$ for each $i\geq 1$, where $c_{r,t}$ is the constant given by 
the induction hypothesis of the lemma (here, the $r$-partite complement $\widetilde{G}_i$ is with respect to the `obvious'  $r$-partition of $G_i$). 
\item Every non-edge of $E$ is covered by $R_1 \cup \dots \cup R_{s'}$.
\end{enumerate}
Suppose that, for $i \in [s']$, the graphs $G_1, \dots , G_i$ have been defined. Apply the induction hypothesis of Lemma~\ref{lemma2} to the $(r-t)$-partite, $K_{r-t}$-free graph $H'_i \cap G_i$ to find a set $R_i \subseteq V(H'_i \cap G_i) \cap (A \cup B)$ with $|I_{\widetilde{G}_i}(R_i)| \geq c_{r,t}|R_i|^{\frac{r-t}{r-t-1}}$ that covers all $(r-t)$-saturating $(A,B)$ edges in $H'_i$. Finally set $G_{i+1} = G_i-R_i$. To check that every non-edge of $E$ is covered by $R_1 \cup \dots \cup R_{s'}$, simply recall that we assumed that every non-edge of $E$ is $(r-t)$-saturating in one of the subgraphs $H_1,\ldots , H_s$ and therefore $(r-t)$-saturating in one of the subgraphs $H'_1,\ldots,H'_{s'}$. Thus, a non-edge $ e \in E$ is $(r-t)$-saturating in some $H'_j$ for some $j \in [s']$, and so it will be covered by one of $R_1,\ldots,R_j$. That is, it will be covered in stage $j$, if it has not been covered already.

To finish the proof of Part 2 of the lemma, we write $R'= R_1 \cup \cdots \cup R_{s'}$. Noting that the sets $R_1,\ldots,R_{s'}$ are pairwise disjoint, we apply H\"{o}lder's inequality to obtain
\[ |R'| \,= \, \sum_{i = 1}^{s'} |R_i| \, \leq \,s'^{\frac{1}{r-t}}\left(\sum_{i=1}^{s'} |R_i|^{\frac{r-t}{r-t-1}}\right)^{\frac{r-t-1}{r-t}} ,   
\] and therefore
\[ s'^{-\frac{1}{r-t-1}}|R'|^{\frac{r-t}{r-t-1}} \, \leq \, \sum_{i=1}^{s'} |R_i|^{\frac{r-t}{r-t-1}} . 
\] Now, since the sets of edges $\{ I_{\widetilde{G}_{i}}(R_i) \}_{i \in [s']}$ are pairwise disjoint (as the sets $R_1, \ldots, R_{s'}$ are pairwise disjoint, and we remove $R_i$ from $G_i$ to define $G_{i+1}$) we may estimate
\begin{align*}
|I_{\widetilde{G}}(R')|  \, = \,  \sum_{i=1}^{s'} |I_{\widetilde{G}_{i}}(R_i)| &\geq\, \sum_{i=1}^{s'} c_{r,t}|R_i|^{\frac{r-t}{r-t-1}} \\
&\geq \, c_{r,t}s'^{-\frac{1}{r-t-1}}|R'|^{\frac{r-t}{r-t-1}} \\
&\geq \, c'_{r,t}s^{-\frac{1}{r-t-1}}|R'|^{\frac{r-t}{r-t-1}}, \end{align*} where $c'_{r,t}$ is a constant depending only on $r, t$. 
Note that the first equality holds since the sets $I_{\widetilde{G}_{i}}(R_i), i \in [s']$ are pairwise disjoint and the sum $\sum_{i=1}^{s'} |I_{\widetilde{G}_{i}}(R_i)|$ counts edges in $\widetilde{G}$ covered by $R'$. This completes the proof of Part~\ref{lemma2:part2} of Lemma~\ref{lemma2}. 

To prove the first part we use the second part along with an extra ingredient.
\paragraph{}
\emph{Proof of Part~\ref{lemma2:part1}  :} We may assume that there is \emph{some} saturating $(A,B)$-edge, otherwise we are trivially done with the choice of 
$R = \emptyset$. So, let $\mathcal{M}$ be a $K_{r-2}$-matching of maximum size in the graph $G[X_1, \dots , X_{r-2}]$ and let $Y$ denote the collection of vertices contained in a clique of $\mathcal{M}$. Note that $\mathcal{M}$ is nonempty as there is some saturating $(A,B)$-edge, and put $L = |\mathcal{M}|$ so that $|Y| = (r-2)L >0$. For each $y \in Y$, let $G(y)$ be the $(r-1)$-partite graph induced on the neighbourhood of $y$ in $G$ with vertex classes $N(y) \cap A , N(y) \cap B$ along with $N(y) \cap X_i$  for $y \not\in X_i, i \in [r-2]$. Our first claim asserts that we may assume there are many non-edges between $Y$ and either $A$ or $B$. 
\begin{claim}\label{claim:many_nonedges}
There are either at least $\frac{1}{4(r-2)}|A||Y|$ non-edges between $Y$ and $A$, or at least $\frac{1}{4(r-2)}|B||Y|$ non-edges between $Y$ and $B$. 
\end{claim}
\begin{proof}
For each $K \in \mathcal{M}$ and $S \subseteq V(G)$ we denote by $d_S(K)$ the number of vertices of $S$ joined to every vertex of $K$, so that $d_S(K) = |N_G(K)\cap S|$. We may assume that, for every $K \in \mathcal{M}$, either $d_{A}(K) \leq \frac{1}{2}|A|$ or $d_{B}(K) \leq \frac{1}{2}|B|$. Indeed, suppose that there is $K \in \mathcal{M}$ with $d_{A}(K) > \frac{1}{2}|A|$ and $d_{B}(K) > \frac{1}{2}|B|$. As $G$ is $K_r$-free we must then count more than $\frac{1}{4}|A||B|$ non-edges between $A$ and $B$. Setting $R$ to be the smaller of $A$ and $B$, we see that trivially $R$ covers all $r$-saturating $(A,B)$ edges and
\[ |I_{\widetilde{G}}(R)|\, >\, \frac{1}{4}|A||B|\, \geq\, \frac{1}{4}|R|^2, \] so we are done (with room to spare). Therefore, we may assume that for every $K \in \mathcal{M}$ either $d_{A}(K) \leq \frac{1}{2}|A|$ or $d_{B}(K) \leq \frac{1}{2}|B|$. 

Write $\mathcal{M}=\mathcal{M}_{A} \cup \mathcal{M}_{B}$, where $\mathcal{M}_{A}$ are those $K\in \mathcal{M}$ which satisfy $d_{A}(K) \leq \frac{1}{2}|A|$ and $\mathcal{M}_{B}$ are those that satisfy $d_{B}(K) \leq \frac{1}{2}|B|$. Then, without loss of generality, we have $|\mathcal{M}_{A}| \geq \frac{1}{2}|\mathcal{M}|$. Now since each $K \in \mathcal{M}_A$ sends at least $\frac{1}{2}|A|$ non-edges to $A$ and since each clique in $\mathcal{M}$ is vertex-disjoint, we have that there are at least $\frac{1}{4}|A||\mathcal{M}|=\frac{1}{4(r-2)}|A||Y|$ non-edges between $Y$ and $A$.
\end{proof}

Now, observe that, by the maximality of $\mathcal{M}$, every $r$-saturating $(A,B)$ edge is $(r-1)$-saturating in one of the graphs $\{ G(y) \}_{ y \in Y }$. Hence we may apply the bound in Part~\ref{lemma2:part2} of the lemma to obtain a set $R_0$ which covers every $r$-saturating $(A,B)$ edge and
\begin{equation}
|I_{\widetilde{G}}(R_0)|  \geq c'_{r,1}(r-2)^{-\frac{1}{r-2}}L^{-\frac{1}{r-2}}|R_0|^{\frac{r-1}{r-2}}. \label{eq: InductionEstimate}
\end{equation}

However, this bound is not useful if $L$ is too large. In order to deal with this issue we shall randomly augment $R_0$ with a set $R'_0$ of $|R_0|$ vertices. The resulting set 
$R = R_0 \cup R_0'$ will only be a factor of two larger than $R_0$ but will cover `many' edges of $\widetilde{G}$ --- enough to achieve a better lower bound on $|I_{\widetilde{G}}(R)|$.

To this end, note that by Claim~\ref{claim:many_nonedges} we may assume that, without loss of generality, there are at least $\frac{1}{4(r-2)}|A||Y|$ non-edges between $Y$ and $A$. Further, we may assume that $|R_0| \leq |A|$. Indeed, suppose otherwise that $|R_0| > |A|$. If $|A||\mathcal{M}| \geq |A|^{\frac{r}{r-1}}$, we are done by choosing $R = A$, since then $|I_{\widetilde{G}}(A)| \geq \frac{1}{4}|A||\mathcal{M}| \geq \frac{1}{4}|A|^{\frac{r}{r-1}}$. Otherwise, $L = |\mathcal{M}| < |A|^{\frac{1}{r-1}} < |R_0|^{\frac{1}{r-1}}$, and using (\ref{eq: InductionEstimate}) yields $|I_{\widetilde{G}}(R_0)| \geq c'_r|R_0|^{\frac{r}{r-1}}$,
so we are done with the choice $R = R_0$.

Hence, assuming that $|R_0| \leq |A|$, by Observation~\ref{obs1}, one can find a subset $R_0^{\prime} \subseteq A $ of size $|R_0|$ such that the number of non-edges between $R_0^{\prime}$ and $Y$ is at least $\frac{1}{4(r-2)}|R_0||Y|= \frac{1}{4}|R_0|L$.

We now set $R = R_0 \cup R_0^{\prime}$ and claim that $R$ is our desired set. First note that $R$ covers all $r$-saturating $(A,B)$ edges in $G$, as $R_0$ already does. To count the total number of non-edges covered by $R$, we note that $|R| \leq 2|R_0|$, and so we have (using~(\ref{eq: InductionEstimate}))
\begin{align} 
2|I_{\widetilde{G}}(R)| &\geq |I_{\widetilde{G}}(R_0)| + |I_{\widetilde{G}}(R'_0)| \notag\\
&\geq c'_{r,1}(r-2)^{-\frac{1}{r-2}}L^{-\frac{1}{r-2}}|R_0|^{\frac{r-1}{r-2}} + \frac{1}{4}|R_0|L \notag \\
&\geq c'L^{-\frac{1}{r-2}}|R|^{\frac{r-1}{r-2}} + \frac{1}{8}|R|L, \label{eqn1}
\end{align}
where $ c' = c'_{r,1}2^{-\frac{r-1}{r-2}}(r-2)^{-\frac{1}{r-2}}$. A simple analysis reveals that the quantity on the right-hand side of (\ref{eqn1}) is minimized in $L$ if $L = (8c'/(r-2))^\frac{r-2}{r-1}|R|^\frac{1}{r-1}$. Substituting this value of $L$ back into (\ref{eqn1}) yields
\[ |I_{\widetilde{G}}(R)|\, \geq \, c_r|R|^\frac{r}{r-1}, 
\] where $c_r$ is a constant depending only on $r$. 
\end{proof}

\subsection{Finishing the proof}\label{subsection_Finishing}
We can now proceed to finish the proof of Theorem~\ref{thm2}.
\begin{proof}[Proof (of Theorem~\ref{thm2})]
Let $r, n$ be integers with $r \ge 2$ and $n \ge 900r^6$, and suppose that $G$ is an $n$-vertex $(r+1)$-saturated graph with $e(G) \geq t_r(n) - m$. For notational convenience we shall write $m = \varepsilon n^2$. Thus we must find a complete $r$-partite subgraph of $G$ on at least $(1-C_r\varepsilon n^{\frac{r-1}{r}})n$ vertices, for some constant $C_r$ depending only on $r$. We shall additionally insist that $C_r \ge 1$. The result is then trivial if $\varepsilon > n^{-\frac{r-1}{r}}$ and so we may assume that $\varepsilon \leq n^{-\frac{r-1}{r}}$. Since $n \geq (30r^3)^2$ we have that $\varepsilon \leq (30r^3)^{-1}$, so we may apply Lemma~\ref{lemma1} to obtain a subset $T \subseteq V(G)$ such that $|T| \leq d_r\varepsilon n$ and $G - T$ is $r$-partite. Let the vertex classes of $G - T$ be $V_1, \dots , V_r$. We now simply apply Part~\ref{lemma2:part2} of Lemma~\ref{lemma2} to common neighbourhoods of appropriate subsets of $T$. But before we do this we need a bound on $e(\widetilde{G}[V_1,\ldots,V_r])$, the number of non-edges between the parts $V_1,\ldots,V_r$, which is the content of the following claim.
\begin{claim}\label{claim:few_nonedges}
$e(\widetilde{G}[V_1,\ldots,V_r]) \leq (d_r+1)\varepsilon n^2$.
\end{claim}
\begin{proof}
First note that if $|T| = 0$, then $G$ is $r$-partite and $e(\widetilde{G}[V_1,\ldots,V_r]) = 0$ since $G$ is $(r+1)$-saturated. So, we may assume that $|T| \geq 1$. In this case, the number of non-edges $e\left(\overline{G}\right)$ satisfies $e\left(\overline{G}\right) \leq {n\choose2} - t_r(n) + \varepsilon n^2$, and also
\[ e\left(\overline{G}\right) \ge \sum_{i=1}^{r}\binom{|V_i|}{2} + e(\widetilde{G}[V_1,\ldots,V_r]) \geq r\binom{\frac{n-|T|}{r}}{2} + e(\widetilde{G}[V_1,\ldots,V_r]), \]
by convexity of the function $x \mapsto \binom{x}{2}$. By using the estimate $t_r(n) \geq \left(1-\frac{1}{r}\right){n\choose2}$, combining the lower and upper bounds on $e\left(\overline{G}\right)$, and rearranging, we get

\begin{align}
e(\widetilde{G}[V_1,\ldots,V_r]) \leq \varepsilon n^2 + \frac{1}{r}{n\choose2} - r\binom{\frac{n-|T|}{r}}{2} &< \varepsilon n^2 + \frac{r-1}{2r}n + \frac{n|T|}{r} \label{eq:nonedgeEstimate1} \\
&= \varepsilon n^2 + 2n|T|\left(\frac{r-1}{4r|T|} + \frac{1}{2r}\right). \label{eq:nonedgeEstimate2}
\end{align}
Now, if $|T| \geq r/2$, then (\ref{eq:nonedgeEstimate2}) is at most $\varepsilon n^2 + \frac{2n|T|}{r}$, and we are done. If $|T| < r/2$, then by (\ref{eq:nonedgeEstimate1}) we have $e(\widetilde{G}[V_1,\ldots,V_r]) < \varepsilon n^2 + n = \left(1 + \frac{1}{\varepsilon n}\right)\varepsilon n^2$. But clearly $\frac{1}{\varepsilon n} \leq d_r$, as otherwise $|T| < 1$. Hence, the desired bound on $e(\widetilde{G}[V_1,\ldots,V_r])$ holds.
\end{proof}

For $t \in [r-1]$ let $\mathcal{C}_t$ denote the collection of copies of $K_t$ contained in $G[T]$, the graph induced on $T$. We say a non-edge $e$ is of \emph{type} $t$ if it lies between two of the classes $V_1,\ldots,V_r$, and the addition of $e$ to $G$ creates a $K_{r+1}$ with exactly $t$ vertices in $T$. Since $G$ is $(r+1)$-saturated and $G[V_1, \ldots , V_r]$ is a $K_{r+1}$-free graph, every non-edge between two of the classes $V_1,\ldots,V_r$ is of type $t$ for some $t \in [r-1]$. For $t \in [r-1]$ we let $E_t$ denote the collection of type $t$ non-edges.

Set $V = V_1 \cup \cdots \cup V_r$ and define $\mathcal{G}_t = \{ G[N(K)\cap V] : K \in \mathcal{C}_t \}$ for $t \in [r-1]$. For each $i \not= j \in [r]$, we show that one can make the induced bipartite graph $G[V_i, V_j]$ complete by removing a relatively small number of vertices. Doing this in succession for each of the $r\choose 2$ pairs $V_i$, $V_j$ with $i \neq j$ then yields a complete $r$-partite subgraph. 

So fix $i \not= j \in [r]$ and note that for each $t \in [r-1] $, each graph in the collection $\mathcal{G}_t$ is $K_{r+1-t}$-free and every $(V_i, V_j)$ non-edge of $E_t$ is $(r+1-t)$-saturating in one of the graphs of $\mathcal{G}_t$. So for each $t \in [r-1]$ we may invoke Part~\ref{lemma2:part2} of Lemma~\ref{lemma2} to obtain a set $S_t(i,j) \subseteq V_i\cup V_j$ that covers every $(r+1)$-saturating $(V_i,V_j)$ edge of type $t$ and \[ \left|I_{\widetilde{G}[V_1,\ldots,V_r]} (S_t(i,j)) \right| \geq  c'_{r, t}|\mathcal{C}_t|^{-\frac{1}{r-t}}|S_t(i,j)|^{\frac{r+1-t}{r-t}}.\] 
Moreover, by Claim~\ref{claim:few_nonedges} we have $|I_{\widetilde{G}[V_1,\ldots,V_r]} (S_t(i,j)) | \leq e(\widetilde{G}[V_1,\ldots,V_r]) \leq (d_r+1)\varepsilon n^2$, and using the bound $|\mathcal{C}_t| \leq |T|^t \le (d_r\varepsilon n)^t$, we obtain
\begin{align*}
|S_t(i,j)|^{\frac{r+1-t}{r-t}} &\leq (c_{r, t}')^{-1}(d_r+1)\varepsilon n^2|\mathcal{C}_t|^{\frac{1}{r-t}}\\
&\le (c_{r, t}')^{-1}(d_r+1)d_r^{\frac{t}{r-t}}\varepsilon^{\frac{r}{r-t}}n^{\frac{2r-t}{r-t}}.
\end{align*}
It follows that
\[
	|S_t(i, j)| \le C_{r,t}\left(\varepsilon^{\frac{r}{r+1-t}}n^{\frac{r-1}{r+1-t}}\right)n,
\]
where $C_{r,t}$ is a constant depending only on $r,t$, for each $t \in [r-1]$, and $i\not=j \in [r]$. 

As every edge between the parts $V_1,\ldots ,V_r$ is of type $t$ for some $t \in [r-1]$, we conclude that the set $S = \bigcup_{t = 1}^{r-1} \bigcup_{i \not= j \in [r]} S_t(i,j)$ covers every non-edge between the parts $V_1,\ldots,V_r$. It follows that $G - S - T$ is a complete $r$-partite graph. To bound $|S|$ recall that $\varepsilon \leq n^{-\frac{r-1}{r}}$. Then we have 
\begin{align*}
|S| &\leq \sum_{t = 1}^{r-1} \sum_{i\not= j \in [r]} C_{r,t}\left(\varepsilon^{\frac{r}{r+1-t}}n^{\frac{r-1}{r+1-t}}\right)n \\
&\leq (r-1)\binom{r}{2}\max_{t \in [r-1]}\{C_{r,t}\}\left(\varepsilon n^{\frac{r-1}{r}}\right)n\\
&\leq C'_r\left(\varepsilon n^{\frac{r-1}{r}}\right)n,
\end{align*}
where the constant $C_r'$ depends only on $r$. It is here that we have used the condition $\varepsilon \leq n^{-\frac{r-1}{r}}$, since this implies that the dominating term in the sum above is the one with $t = 1$. Hence we have found a complete $r$-partite subgraph on 
\[ n - |S| - |T| \geq n - C'_r\left(\varepsilon n^{\frac{r-1}{r}}\right)n - d_r\varepsilon n \geq \left( 1- C_r \varepsilon n^{\frac{r-1}{r}}\right) n\]
vertices, for some constant $C_r$. This completes the proof. \end{proof}

\section{Constructions}\label{sec:constructions}
\label{section_constructions}\label{subsec:aux_constructions}
\subsection{An auxiliary construction: removed edges}
The aim of this section is to describe a family of constructions that demonstrate the optimality of Theorem~\ref{thm:MainTheoremSoftForm}. 
We begin by inductively constructing a family of auxiliary graphs $G_{r,s}$, for each $r, s \in \mathbb{N}$, $r, s \geq 2$. 
It is useful to keep in mind that the edges of the $r$-partite graph $G_{r,s}$ record edges to be removed from a later graph.
First let us introduce a family of $r$-partite graphs $G_{r, s_{1}, s_{2}, \dots, s_{r-1}}$ for which $G_{r, s}$ will be a special 
case.

    \emph{Construction of $G_{r, s_{1}, \dots, s_{r-1}}$ :}
    Let $s_{1}, \dots, s_{r-1} \geq 2$ be integers. We define a sequence of graphs $G_{2,s_1},G_{3,s_1,s_2},\ldots ,G_{r,s_1,\ldots,s_{r-1}}$ inductively,
    where $G_{i,s_1,\ldots,s_{i-1}}$ will be an $i$-partite graph.
    First, we define $G_{2, s_{1}}$ to be the complete bipartite graph $K_{s_{1}, s_{1}}$.
    Now let $2 \leq t \leq r-1$ and assume that we have defined the $t$-partite graph $G_{t, s_{1}, \dots, s_{t-1}}$. We define $G_{t+1, s_{1}, \dots, s_{t}}$ as follows.
    Let $ H_1,\ldots, H_{s_{t}}$ be vertex disjoint copies of $G_{t, s_{1}, \dots, s_{t-1}}$ and suppose $H_p$ has vertex classes $A^p_1,\ldots, A^p_t$, for each $p \in [s_{t}]$. 
    Define $G_{t+1, s_{1}, \dots, s_{t}}$ to be the $(t+1)$-partite graph with the first $t$ vertex classes defined as $A_i := A_i^1 \cup \cdots \cup A_i^{s_{t}}$, for $i \in [t]$, and with the $(t+1)$st vertex class defined as a collection of new vertices $A_{t+1} = \{ x_1,\ldots,x_{s_{t}} \}$. 
    We define the edge set $E(G_{t+1, s_{1}, \dots, s_t}) = \bigcup_{p=1}^{s_{t}} E(H_{p}) \cup \left\{ x_{p}y: y \in H_{q}, p, q \in [s_{t}], p \neq q \right\}$.

Now let $G_{r, s} = G_{r, 2s, s, \dots, s }$, for $s \ge 2$.
This choice of the parameters $s_1, \ldots , s_{r-1}$ is optimal for its use later in the construction; however, for brevity, we shall omit this calculation.
The following proposition records several useful properties of our family of graphs $G_{r,s}$.

\begin{prop}
    \label{prop_construction}
    The graph $G_{r, s}$ has the following properties.
    \begin{enumerate}
        \item \label{prop:ConstructionRPartite} $G_{r, s}$ is $r$-partite with vertex partition $A_1 \cup \cdots \cup A_r$ (and hence it makes sense to consider the $r$-partite complement of $G_{r,s}$ with respect to this partition).
        \item \label{prop:ConstructionKrFree} The $r$-partite complement $\widetilde{G}_{r,s}$ is $K_{r}$-free.
        \item \label{prop:ConstrKr-1} For each $i \in [r]$, there is a copy of $K_{r-1}$ in $\widetilde{G}_{r,s} \setminus A_{i}$.
        \item \label{prop:ConstrSat} Every edge between two different vertex classes of $G_{r,s}$ is $r$-saturating in $\widetilde{G}_{r, s}$.
        \item \label{prop:ConstVert} $|G_{r,s}| = \sum_{i=1}^{r-2}s^{i} + 4s^{r-1} = \frac{s}{s-1}(4s^{r-1}-3s^{r-2}-1) \le 4\frac{s^{r}}{s-1}$.
        \item \label{prop:ConstEdges} $e(G_{r,s}) \leq  4(r-1)s^{r}$.
        \item \label{prop:ConstClassSize} The size of the largest two vertex classes is $2s^{r-1}$.
        \item \label{prop:SmallerClassSize} All other vertex classes have size at most $s^{r-2}$.
        \item \label{prop:ConstMatching} There is a matching between the largest two vertex classes of $G_{r,s}$.
        \item \label{prop:ConstIndSet} Any independent set in $G_{r, s}$ has at most $|G_{r,s}| - 2s^{r-1}$ vertices.
    \end{enumerate}
\end{prop}
\begin{proof}
    We shall use induction on $r$.
    The base case $r=2$ is trivial.
    Suppose the assertions hold for $r \ge 2$. 
    Clearly $G_{r+1, s}$ is $(r+1)$-partite and $\widetilde{G}_{r+1,s}$ is $K_{r+1}$-free.
    To show Part~\ref{prop:ConstrKr-1}, suppose first that $i = r+1$.
    By induction hypothesis there is a copy of $K_{r-1}$ in $\widetilde{H}_{1} \setminus A_{1} = \widetilde{H}_{1} \setminus A_{1}^{1}$ which together with any $x \in \widetilde{H}_{2} \cap A_{1}$ form a copy of $K_{r}$ in $\widetilde{G}_{r+1, s} \setminus A_{r+1}$.
    The argument is very similar for the case when $i \in [r]$.
    To show Part~\ref{prop:ConstrSat}, notice that the only edges between vertex classes in $G_{r+1, s}$ are either inside $H_{p}$ or between $x_{p}$ and $H_{q}$, for some $p, q \in [s]$, $p \neq q$.
    If we add an edge to $\widetilde{G}_{r+1,s}$ (which corresponds to removing that edge from $G_{r+1, s}$) of the former type, the assertion holds simply by induction.
    If we add an edge $x_py$ with $y \in A_{i}$, $i \in [r]$, of the latter type, first observe that it follows from Part~\ref{prop:ConstrKr-1} of the induction hypothesis that $\widetilde{H}_{p} \setminus A_{i}$ contains a copy of $K_{r-1}$, say $K$.
    Hence, both $x_{p}$ and $y$ are joined to every vertex in $K$, thus forming a $K_{r+1}$ in $\widetilde{G}_{r+1,s}$.  
    The number of vertices satisfies the relation $|G_{r+1,s}| = s + s|G_{r, s}|$ while $|G_{2,s}| = 4s$, and thus the claim follows. 
    The number of edges satisfies the recurrence $e(G_{r+1, s}) = s\cdot e(G_{r,s}) + s(s-1)|G_{r, s}| \le s\cdot e(G_{r,s}) + s(s-1)4\frac{s^{r}}{s-1} = s\cdot e(G_{r,s}) + 4s^{r+1}$ so, by induction, $e(G_{r+1, s}) \leq 4(r-1)s^{r+1} + 4s^{r+1} = 4rs^{r+1}$. Parts~\ref{prop:ConstClassSize},\ref{prop:SmallerClassSize},\ref{prop:ConstMatching} follow immediately by induction. 
    Finally, to argue Part~\ref{prop:ConstIndSet}, simply notice that for each $p \in [s]$, by induction, there is no independent set in $H_{p}$ with more than $|H_{p}| - 2s^{r-1}$ vertices. 
    Therefore, from disjointness of the $H_{p}$'s, any independent set in $G_{r+1,s}$ has at most $|G_{r+1,s}| - 2s^{r}$ vertices.
\end{proof}

\subsection{The final construction}\label{subsec:final_construction}
We can now proceed to construct a family of graphs $H_{r, s, t}(n)$ that will demonstrate the tightness of Theorem~\ref{thm2}.
We let $H_1,\ldots,H_t$ be vertex disjoint copies of $G_{r,s}$ with vertex partitions $H_p = A_1^p \cup \cdots \cup A_r^p$ for each $p \in [t]$. 
    We now augment the vertex set of the $H_p$'s to be the vertex set for our $G$. 
    First note that since $n \geq 4s^{r-1} t r + t$, we can find $\ell_1,\ldots,\ell_r \in \mathbb{N}$, so that for each $i \in [r]$ we have $\sum_{p=1}^{t} |A^p_i| + \ell_i \in \left\lbrace \lfloor \frac{n-t}{r} \rfloor, \lceil \frac{n-t}{r} \rceil \right\rbrace$ and $\sum_{i=1}^{r} \left(\sum_{p=1}^{t} |A^p_i| + \ell_i\right) =n-t$. 
    Note that as $n$ is large enough, we may assume that $\ell_1, \ldots, \ell_r >0$. 
    We now define the sets $A_1,\ldots,A_r$ as
    \[ 
        A_i = A^1_i \cup \cdots \cup A^t_i \cup Y_i, 
    \] 
    for $i \in [r]$, where $Y_i$ is a collection of $\ell_i$ new vertices.
    We additionally define $A_{r+1} = \{x_1,\ldots,x_t\}$ as a collection of $t$ new vertices and finally set $V(G)  = \bigcup_{i=1}^{r+1} A_i$. 

    We define the edge set as follows: the vertex $x_p$ is joined to $V(H_p)$, for each $p \in [t]$, and for $i,j \in [r]$, $x \in A_i, y \in A_j$, $xy$ is an edge if and only if $i \not= j$ and the edge $xy$ is \emph{not} in any of the graphs $H_1,\ldots,H_t$. 
    We then add a maximal set of edges among $A_{r+1}$ that leaves the graph $K_{r+1}$-free.
    That is, we first define a graph $G'$ by $V(G') = V(G)$ and
    \[
        E(G') = \{ x_py : y \in V(H_p), p \in [t] \} \cup \left\lbrace xy : x \in A_i, y \in A_j \, , 1\leq i < j \leq r \right\rbrace \setminus \bigcup_{p = 1}^t E(H_p), 
    \]
    and then augment the edge set to form $E(G)$:
    \[ 
        E(G) = E(G') \cup X ,
    \] 
    where $X \subseteq A_{r+1}^{(2)}$ is maximal in the sense that adding any further edge of $A_{r+1}^{(2)}$ will yield a $K_{r+1}$ in $G$.
    Call this final graph $H_{r, s, t}(n)$.
   
   The following Proposition shows that  $H_{r, s, t}(n)$ has all of the properties that are of interest to us. 
    Before proceeding, let us note the following easy observation.
    
    \begin{observation}\label{obs:turan}
     For integers $r, t \le n$ with $r \ge 2$ we have
    \[
    	t_r(n - t) \ge t_r(n) - (1- 1/r)tn.
	\]
\end{observation}
	\begin{proof}
	If $x$ is a vertex of minimum degree in $T_r(n)$, then $T_r(n) - x = T_r(n-1)$ and so $t_r(n) = t_r(n-1) + \delta(T_r(n))$.
	Iterating this fact yields
	\[
		t_r(n) = t_r(n-t) + \sum_{j=0}^{t-1}\delta(T_r(n-j)) \le t_r(n-t) + t\cdot\delta(T_r(n)) \le t_r(n-t) + (1 - 1/r)tn,
	\]
	as claimed.
	\end{proof}
	
	\begin{prop} \label{main_construction}  
    Suppose that $n,r,s,t \in \mathbb{N}$ with $r,s \geq 2 $ satisfy $n \geq 4s^{r-1}tr + t$. 
    Then there exists an $(r+1)$-saturated graph $G$ on $n$ vertices with $e(G) \geq t_r(n) - \frac{r-1}{r}tn - 4(r-1)ts^r$ such that any complete $r$-partite subgraph has at most $n - 2ts^{r-1}$ vertices.
\end{prop} 
\begin{proof}
    Let $G = H_{r, s, t}(n)$. We see that $G$ satisfies
    \begin{align*}
        e(G) \geq t_{r}(n-t) - t\cdot e(G_{r, s}) &\ge t_r(n) - \frac{r-1}{r}tn - t \cdot e(G_{r,s})\\ &\geq t_r(n) - \frac{r-1}{r}tn - 4(r-1)ts^r,
    \end{align*}
	where in the second inequality we have used Observation~\ref{obs:turan}.
    We first note that any complete $r$-partite subgraph is of order at most $n - 2ts^{r-1}$, as for each $p \in [t]$, at most $|H_p| - 2s^{r-1}$ vertices from $V(H_p)$ can be included in a complete $r$-partite subgraph of $G$, by Part~\ref{prop:ConstIndSet} of Proposition~\ref{prop_construction}.

    To see that $G$ is $(r+1)$-saturated we may argue as we did in the proof of Proposition~\ref{prop_construction}. 
    First, notice that $G$ is $K_{r+1}$-free.
    Indeed, if there were a copy of $K_{r+1}$ in $G$ then, by construction, it would contain exactly one vertex from $A_{r+1}$, say $x_{p}$ for some $p \in [t]$.
    Since the neighbourhood of $x_{p}$ outside $A_{r+1}$ is exactly $H_{p}$, which is $K_{r}$-free, it follows that $x_{p}$ cannot be contained in any copy of $K_{r+1}$, which yields a contradiction.
    There are only three types of edges that one could add to $G$: edges from $E(H_p)$, for some $p \in [t]$, edges between $A_{r+1}$ and one of the $A_i$, $i \in [r]$, and edges within a vertex class. 
    Note that the first option must create a $K_{r}$ by Proposition~\ref{prop_construction}, which then extends to a $K_{r+1}$ when we include $x_p$. 
    If we add an edge $x_py$, for some $y \in A_{i}$, $p \in [t]$, $i \in [r]$, first notice that by Part~\ref{prop:ConstrKr-1} of Proposition~\ref{prop_construction} we may choose a copy of $K_{r-1}$, say $K$, in the graph induced on $V(H_p) \setminus A_{i}$.
    We then form a $K_{r+1}$ by observing that $x_p$ and $y$ are joined to all of $K$. 
    If we add an edge within one of the classes $A_1,\ldots,A_r$, then we find a $K_{r-1}$ among $Y_1,\ldots,Y_r$ that does not intersect the class that contains the added edge.
    Clearly this $K_{r-1}$ is in the common neighbourhood of both points of the added edge and hence we extend to a $K_{r+1}$.
    Adding an edge within $A_{r+1}$ guarantees a $K_{r+1}$ by the construction of $G$.
\end{proof}

By choosing $s$ and $t$ appropriately, we arrive at the following.

\begin{theorem}\label{thm: construction}
	Let $r \ge 2$ be an integer and let $\varepsilon > 0$. Then 
    there exist $n_0, b_0, c_0 > 0$ which are constants depending on $r$ and $\varepsilon$ such that the following holds.
    Let $n \in \mathbb{N}$ and $m > 0$ be such that $n \geq n_0$ and $(\frac{r-1}{r} + \varepsilon)n \le m \le b_0n^{\frac{r+1}{r}}$.
    Then there exists an $(r+1)$-saturated graph $G$ on $n$ vertices and $e(G) \geq t_r(n) - m$, with no complete $r$-partite subgraph on more than $\left(1 - c_0mn^{-\frac{r+1}{r}}\right)n$ vertices.
\end{theorem}
\begin{proof}
    Fix any $\varepsilon > 0$ and let $c' = \left( \frac{r-1}{r}+\varepsilon \right)^{-1}$, $c = \varepsilon /4(r-1)$.
    We set $s = \left\lfloor \left(cn\right)^{\frac{1}{r}} \right\rfloor$ and $t = \left\lfloor c'mn^{-1} \right\rfloor$ in Proposition~\ref{main_construction}. 
    Observe that $t \ge 1$, and as long as $ n \ge n_0  \ge \frac{2^{r}}{c}$, then $s \ge 2$.
    It is easy to check that for $b_{0} \le \left(8c'c^{\frac{r-1}{r}}r\right)^{-1}$ the condition $n \geq 4s^{r-1}tr + t$ holds for this choice of $s$ and $t$. 
    Indeed, we have:
    \begin{align*}
        4s^{r-1}tr + t 
        &\le 8s^{r-1}tr \le 8 \left(c n\right)^{\frac{r-1}{r}}rc'mn^{-1} \le 8(c n)^{\frac{r-1}{r}}rc'b_{0}n^{\frac{r+1}{r}}n^{-1} \\
        &= 8c'c^{\frac{r-1}{r}}rb_{0}n \le n,
    \end{align*}
    where the penultimate inequality follows from the assumption that $m \le b_{0}n^{\frac{r+1}{r}}$.
    
    Let $G$ be as in the conclusion of Proposition~\ref{main_construction}.
    It follows that
    \begin{align*}
        e(G) 
        &\ge t_{r}(n) - \frac{r-1}{r}tn - 4(r-1)ts^{r} \geq t_{r}(n) - \frac{r-1}{r}tn - 4(r-1)c tn\\
        &\ge t_{r}(n) - tn\left( \frac{r-1}{r}+4(r-1)c \right) \ge t_{r}(n) - m.
    \end{align*}

    To finish the proof, notice that $t \ge \frac{c'mn^{-1}}{2}$ and $s \ge \frac{(c n)^{\frac{1}{r}}}{2}$.
    Therefore we have
    \begin{align*}
        g_{r}(G) \ge 2ts^{r-1} \ge 2\frac{c'mn^{-1}}{2}\frac{(c n)^{\frac{r-1}{r}}}{2^{r-1}} \ge \frac{c'c^{\frac{r-1}{r}}}{2^{r-1}}mn^{-\frac{1}{r}}.
   \end{align*}
   Hence we have that there is no complete $r$-partite subgraph on more than $(1-c_0mn^{-\frac{r+1}{r}})n$ vertices, where $c_0 = \frac{c'c^{\frac{r-1}{r}}}{2^{r-1}}$. 
    
    
\end{proof}

\section{Beyond the threshold: $(r+1)$-saturated graphs on $t_r(n) - O(n^{\frac{r+1}{r}})$ edges}\label{balanced}

If $G$ is an $(r+1)$-saturated graph with $t_r(n)-o(n^{\frac{r+1}{r}})$ edges, then Theorem~\ref{thm:MainTheoremSoftForm} tells us that $G$ has a complete $r$-partite subgraph $G' = V_1 \cup \dots \cup V_r$ on $(1- o(1))n$ vertices. It is easy to see that no two classes $V_i, V_j$ can differ by more than $o(n)$ vertices (otherwise, there would be too few edges in $G$), and so we may remove at most $o(n)$ vertices to make $G'$ a $r$-partite Tur\'{a}n graph. In other words,
there is little quantitative difference between the maximum sized $r$-partite Tur\'{a}n subgraph and the maximum sized complete $r$-partite subgraph in the edge regime $t_r(n)-o(n^{\frac{r+1}{r}})$. However, if $e(G) = t_r(n) - O(n^{\frac{r+1}{r}})$ the difference between these two problems becomes relevant, and we find it most natural to restrict our attention to \emph{balanced} complete $r$-partite subgraphs or, equivalently, $r$-partite Tur\'{a}n graphs.
\paragraph{}
Recall that for $n,m \in \mathbb{N}$, the quantity $g^*_r(n,m)$ is the maximum number vertices that one must remove from an $(r+1)$-saturated graph on $t_r(n) - m$ edges so that the remaining graph is an $r$-partite Tur\'{a}n graph. In this section, we show that, for $C$ sufficiently large compared to $r$, we have  
\[ g_r^*(n,Cn^{\frac{r+1}{r}}) \geq \left(1-\frac{c' \log(Cr)}{C}\right)n,
\] for an absolute constant $c'$ and sufficiently large $n$. In other words, the vertex set of the largest $r$-partite Tur\'{a}n subgraph can cover an arbitrarily small fraction of the vertices in the edge range $e(G) =  t_r(n) - O(n^{\frac{r+1}{r}})$. We remind the reader of the statement of Theorem~\ref{thm:RandomConstruction} for convenience.
\begingroup
\def\thetheorem{\ref{thm:RandomConstruction}}
\begin{theorem}
Let $r \geq 2 $ be an integer and let $\delta > 0$. There exists a constant $C = C(r,\delta)$ such that, for $n$ sufficiently large, there exists an $n$-vertex $(r+1)$-saturated graph $G$ that contains no $T_r(\delta r n)$ and $e(G) \geq t_r(n) - Cn^{\frac{r+1}{r}}$. In terms of the function $g_r^*$, we show that, for sufficiently large $D >0$, we have
\[ g_r^*(n,Dn^{\frac{r+1}{r}}) \geq \left(1-\frac{c'\log(Dr)}{D}\right)n , \] for sufficiently large $n$.
\end{theorem}
\endgroup
\begin{proof}

Fix $\delta \in (0, 1)$ and choose $C(\delta) = 2^6 r^{-1}\delta^{-1} \log(2e/\delta) = 4rB(\delta)$, where we have set $B(\delta) = 16 r^{-2}\delta^{-1}\log(2e/\delta) $. With foresight, we select $s = n^{\frac{1}{r}}$, $t = B(\delta)n^{\frac{1}{r}}$ and note that for large enough $n$ we have $\frac{\delta}{4} n > 2s^{r-1}$. 

We build our desired graph $G$ in three stages. We start by defining our first stage graph $G_{I}$. Let $T_r(n-t)$ be the Tur\'{a}n graph on $n-t$ vertices with vertex classes $V_1,\ldots,V_r$ and let $V_{r+1} = \{x_1,\ldots,x_t\}$ be a set of vertices disjoint from $V(T_r(n-t))$. Define $V(G_{I}) = V(T_r(n-t)) \cup V_{r+1}$ and $E(G_{I}) = E(T_r(n-t))$. In the second stage, we use a probabilistic construction to form the graph $G_{II}$ by removing edges between the classes $V_i,V_j$, where $i,j  \in [r]$, and adding edges between the classes $V_{r+1},V_i$, $i \in [r]$. After this second stage we will almost be finished: $G_{II}$ will be a $K_{r+1}$-free graph with many edges; $G_{II}$ will not contain a $T_r(\delta rn)$; and adding non-saturating edges to $G_{II}$ will not ruin these properties. In the final stage we augment $G_{II}$ by choosing an arbitrary maximal $K_{r+1}$-free graph which contains $G_{II}$. This will serve as our final graph $G$.

We now prepare for the second stage. For each $i \in [r]$, fix a vertex $v_i \in V_i$ and then define $V'_i = V_i \setminus \{ v_i \}$. The edges incident to the vertices $v_1,\ldots,v_r$ will go unaltered throughout this construction. This is to ensure that the addition of any edge within any of the classes $V_1,\ldots,V_r$ creates a $K_{r+1}$ with $v_1,\ldots,v_r$, even after the edge deletions in stage II. We now define an auxiliary graph $H$ on $V'_1,\ldots,V'_r$ which records edges that we shall delete from $T_r(n-t)$ to form $G_{II}$. 

For $p \in [t]$, let $H^p$ be a copy of the $r$-partite graph $G_{r,s}$, as defined in Section~\ref{section_constructions}, where we think of the vertex sets of the $H^p$ as being disjoint and $H^p_1,H^p_2$ as being the two largest vertex classes (each of order $2s^{r-1}$) in the vertex partition $H^p_1, \ldots , H^p_r$.

We shall randomly embed each $H^p$ into $V'_1 \cup \cdots \cup V'_r$ in a manner that respects the partition $V'_1,\ldots,V'_r$. 
To this end, we define a probability space on tuples of injections $(f_1,\ldots,f_t)$ with $f_p : V(H^p) \rightarrow \bigcup_{i=1}^{r} V'_i$.
We choose each $f_p$ so that $\{f_p(H^p_{i})\}_p$ are (fixed) vertex disjoint sets for each $3 \leq i \leq r$, while for $i\in \{1,2\}$, $f_p(H^p_i)$ is a uniform random subset of $V_i'$ of size $|H^p_i|$ and each $\{f_p(H^p_i) : i \in \{1,2\}, p \in [t]\}$ is chosen independently.
Note that since $|H^p_3|,\ldots,|H^p_r| \leq s^{r-2}$ (by Part~\ref{prop:SmallerClassSize} of Proposition~\ref{prop_construction}) it is indeed possible to request that $|H^1_i|,\ldots,|H^t_i|$ are disjoint subsets of $V'_i$, as $s^{r-2}t = B(\delta)n^{1-1/r} < (n-1-t)/r$, for large enough $n$.

Define the graph $H(f_1,\ldots,f_t)$ to have vertex set $V'_1 \cup \cdots \cup V'_r$ and edge set
\[
E(H(f_1,\ldots,f_t)) = \bigcup_{p \in [t]} \{ xy : f_p^{-1}(x)f_p^{-1}(y) \in E(H^p) \}.
\]
We define $G(f_1,\ldots,f_t)$ to be a graph on the same vertex set as $G_{I}$ and with edge set 
\[ E(G(f_1,\ldots,f_t)) = \{x_py : y \in f_p(H^p),\ p \in [t]  \} \cup  E(T_r(n-t)) \setminus E(H(f_1,\ldots,f_t)). 
\] In what follows, we show that the probability of making a  ``good'' choice for $G(f_1,\ldots,f_t)$ is non-zero. 




\begin{claim} \label{Claim:GStaysKrfree} Let $f_1,\ldots,f_t$ be any functions as described above. 
The graph $G(f_1,\ldots,f_t)$ is $K_{r+1}$-free.
\end{claim}
\emph{Proof of Claim~\ref{Claim:GStaysKrfree}:} If a copy of $K_{r+1}$ is contained in $G(f_1,\ldots,f_t)$, it must have exactly one vertex in each class $V_1,\ldots,V_{r+1}$. 
Hence there must exist $p \in [t]$ so that $G(f_1,\ldots,f_t)$ induced on $f_p(V(H^p))$ contains a copy of $K_r$. 
This induced graph is contained in a copy of $\widetilde{G}_{r,s}$ (as in Proposition~\ref{prop_construction}), which is $K_r$-free, a contradiction. \qed\vspace{4mm}

We now show that every ``missing'' edge between $V_1,V_2$ are saturating edges. This is important as we need to ensure that the edges we remove in stage II are not just added back in, in the final stage. 

\begin{claim} \label{Claim:Stage2Sat} Let $f_1,\ldots,f_t$ be functions as described above.
Adding any edge, which is not already present, between the classes $V_1,V_2$ in $G(f_1,\ldots,f_t)$ creates a $K_{r+1}$.
\end{claim}
\emph{Proof of Claim~\ref{Claim:Stage2Sat}:} 
Suppose that $e \not\in E_{G(f_1,\ldots,f_t)}(V_1,V_2)$. 
This means that $e \in E_{H(f_1,\ldots, f_t)}(V_1,V_2)$ and thus $e \in E_{f_p(H^p)}(V_1,V_2)$, for some $p \in [t]$.
Every such edge in $f_p(H^p)$, if deleted from $H^p$, is contained in an independent set $I$ with exactly one vertex in each part $V_1,\ldots,V_{r}$; this holds by Part~\ref{prop:ConstrSat} in Proposition~\ref{prop_construction}. Since each of the $H^1,\ldots,H^t$ are disjoint on $V_3,\ldots,V_r$, $I$ is a set containing only $e$, in $H$. This is the same as saying that $e$ is a $r$-saturating edge in $G(f_1,\ldots,f_t)$ in the graph induced on $f_p(V(H^p))$. 
Since the vertex $x_p\in V_{r+1}$ joins to all of $f_p(V(H^p))$, $e$ is $(r+1)$-saturating in $G(f_1,\ldots,f_t) $.\qed\vspace{4mm}

The following claim will help us show that we cannot find a large $r$-partite Tur{\'a}n graph in our final graph. 

\begin{claim} \label{Claim:NoBigBipartite} The probability that $G(f_1,\ldots,f_t)$ contains a complete bipartite graph $K_{\delta n/2,\delta n /2 }$ between $V_1,V_2$ is less than $1/2$.
\end{claim}
\emph{Proof of Claim~\ref{Claim:NoBigBipartite}:}
 Let $E(A, B)$ be the ``bad'' event that the pair $A \subset V'_1$, $B \subset V'_2$ have no edge of $H(f_1,\ldots,f_t)$ between them. We define the random variable $X$ to be the number of pairs of subsets $A \subset V'_1$, $B \subset V'_2$ of size $\delta n/2$ each, that have no edge of $H(f_1, \ldots, f_t)$ between them. 

To estimate the expectation of $X$ we fix two sets $A \subseteq V'_1$, $B \subseteq V'_2$ of size $\delta n/2$, and let $E_p = E_p(A,B)$, for $p \in [t]$, denote the event that $f_p(H^p)$ has no edge between $A,B$.
By independence, $\mathbb{P}(E(A,B)) = \prod_p \mathbb{P}(E_p)$.
We fix $p \in [t]$ and look to bound $\mathbb{P}(E_p)$.
We explicitly express the two largest vertex classes of $H^p$, $H^p_1 = \{ y_1,\ldots,y_{2s^{r-1}} \}$, and $H^p_2 = \{ z_1,\ldots,z_{2s^{r-1}} \}$, where $ y_iz_i $, $i \in [2s^{r-1}]$, are the edges of a perfect matching in $H^p$ between the two largest classes (which is guaranteed by Proposition~\ref{prop_construction}). 
For ease of notation, let $f = f_p$ and let us say that a pair $f(y_i),f(z_i)$ \emph{hits} $A,B$ if $f(y_i) \in A$ and $f(z_i) \in B$.
We will say that $f(y_i),f(z_i)$ \emph{misses} the pair, otherwise. 
We define $E_p(i)$ to be the event that $f(y_i),f(z_i)$ misses $A,B$.  

Note that $\mathbb{P}(E_p)$ is at most
\begin{equation} \label{equ:sec4bound} \mathbb{P}\left( \bigcap_{i=1}^{2s^{r-1}} E_p(i) \right) = \prod_{i=1}^{2s^{r-1}} \mathbb{P}\left(E_p(i)| E_p(i-1),\ldots,E_p(1) \right) .\end{equation} 
So to bound $\mathbb{P}(E_p)$, we need only to bound the terms in the above product. 
This is easily done as the conditional probabilities $\mathbb{P}\left(E_p(i)| E_p(i-1),\ldots,E_p(1) \right)$ do not differ too much from the unconditioned probabilities $\mathbb{P}(E_p(i))$.
To this end, note that $E_p(1),\ldots,E_p(i-1)$ depend only on the choices of $Y_{i-1} = \{ f(y_1),\ldots,f(y_{i-1}) \}, Z_{i-1} = \{ f(z_1),\ldots,f(z_{i-1})\} $.
Thus, we have 
\begin{align*}
 \mathbb{P}\left( E_p(i)| E_p(i-1),\ldots,E_p(1) \right) &\leq \max_{Y_{i-1},Z_{i-1}} \mathbb{P}\left(E_p(i)| Y_{i-1},Z_{i-1} \right) \\
 &= 1 - \min_{Y_{i-1},Z_{i-1}} \mathbb{P}\left(f(y_i),f(z_i) \text{ hits } A,B | Y_{i-1},Z_{i-1} \right)\\
 &\leq 1 - \min_{Y_{i-1},Z_{i-1}} \frac{|A \setminus Y_{i-1} ||B \setminus Z_{i-1} |}{ (|V_1'| - (i-1))^2 } \\
 &\leq 1 -  \left( \frac{ r(\delta n/2 - 2 s^{r-1}) }{ n } \right)^2 \\
 &\leq \exp(-\delta^2r^2/16) , \end{align*}
 
where the third inequality follows by recalling that $|V'_1|,|V'_2| \leq n/r$ and the last inequality follows by recalling that $\frac{\delta}{4} n > 2s^{r-1}$. So, from (\ref{equ:sec4bound}), we have 
\[ \mathbb{P}(E_p) \leq  \exp(-r^2\delta^2s^{r-1}/8), \] for each $p \in [t]$, and therefore
\[ \mathbb{P}(E(A,B)) = \prod_{p=1}^t \mathbb{P}(E_p) \leq \exp \left( - \frac{r^2\delta^2s^{r-1}t}{8} \right). \] So, by linearity of expectation, we have
\[
 \mathbb{E} X  \leq \binom{n}{\delta n/2}^2 \exp\left( - \frac{r^2\delta^2 s^{r-1}t}{8} \right).
\]

Using the standard inequality $\binom{n}{k} \le \left(\frac{ne}{k}\right)^k$, we have
\[
	\mathbb{E}X \le (2e/\delta)^{\delta n}\exp\left( - \frac{r^2\delta^2 s^{r-1}t}{8} \right) =  \exp\left(\delta n \log(2e/\delta) - \frac{r^2\delta^2 s^{r-1}t}{8}\right).
\]
Recalling our choices of $s = n^{1/r}$ and $t = B(\delta)n^{1/r} = 16r^{-2}\delta^{-1} \log(2e/ \delta) n^{1/r}$, we have $\mathbb{E} X < 1/2$ for sufficiently large $n$. This completes the proof of Claim~\ref{Claim:NoBigBipartite}. \qed\vspace{4mm}

We now define $G_{II}$ to be a graph of the form $G(f_1,\ldots,f_t)$ for which there are no copies of $K_{\delta n/2, \delta n/2}$ between vertex classes
$V'_1,V'_2$. Such a graph $G(f_1,\ldots,f_t)$ exists with non-zero probability, by Claim~\ref{Claim:NoBigBipartite}.

To define our final graph $G$, we choose a maximal $K_{r+1}$-free graph which contains $G_{II}$. Since $G_{II}$ is $K_{r+1}$-free, $G$ is also $K_{r+1}$-free and, trivially, $G$ is $(r+1)$-saturated. Using inequalities $t_r(n-t) \ge t_r(n) - tn$ and $e(G_{r, s}) \le 4(r-1)s^r$, we have that
\begin{align*}
 e(G) &\geq e(G_{II}) \\ 
 &\geq t_r(n) -  tn - te(G_{r,s}) \\
 &\geq t_r(n) - B(\delta)n^{\frac{r+1}{r}} - 4(r-1)s^rt \\ 
 &\geq t_r(n) - 4rB(\delta)n^{\frac{r+1}{r}} = t_r(n) - C(r,\delta)n^{\frac{r+1}{r}}.
 \end{align*}
We now observe that $G$ cannot contain a copy $T$ of $T_r(\delta r n)$. Suppose, towards a contradiction, that 
$G$ contains $T$. First note that $G - V_{r+1}$ is $r$-partite with vertex partition
$V_1,\ldots,V_r$. This is because the addition of any pair $e = uv$ to $G_{II}$, \emph{within} some $V_i$, would form a copy of $K_{r+1}$ on vertex set 
$\{u,v\} \cup (\{v_1,\ldots,v_r \} - \{v_i\})$. Therefore $G-V_{r+1}$ is $r$-partite. This means that
$G$ must contain a copy $K$ of $K_{\delta n/2,\delta n/2}$ between $V'_1,V'_2$, as $|T \cap V_{r+1}| \leq |V_{r+1}| = t < \delta n/4$ and therefore $|T \cap (V_1\cup \cdots \cup V_r)| \geq \delta rn/2$. 
Now since all non-edges between $V_1,V_2$ are $(r+1)$-saturating in $G_{II}$ (Claim~\ref{Claim:Stage2Sat}), we have that no edges were added between 
$V_1,V_2$ in forming $G$. 
In other words, $G_{II}[V_1,V_2] = G[V_1,V_2]$. 
This implies that $K \cong K_{\delta n/2, \delta n/2}$ is also a subgraph of $G_{II}$, which contradicts Claim~\ref{Claim:NoBigBipartite}. 
This completes the proof of Theorem~\ref{thm:RandomConstruction}.
\end{proof}



\section{Final Remarks and Open Problems} \label{final_remarks}

Recall that $g_r(n,m)$ is defined to be the maximum number of vertices that one is required to remove from an $n$-vertex, $(r+1)$-saturated graph with at least $t_r(n) - m$ edges, so that the remaining graph is complete $r$-partite. 
Combining Theorems~\ref{thm2} and~\ref{thm: construction} we see that for any $\varepsilon > 0$, if $n \ge n_0(r, \varepsilon)$ and
 $(\frac{r-1}{r}+\varepsilon)n \le m  \le n^{\frac{r+1}{r}}$ one has
\[ c_{r, \varepsilon}mn^{-1/r} \leq g_r(n, m) \leq C_rmn^{-1/r}, \]
where $c_{r,\varepsilon}$ depends on $r, \varepsilon$, and $C_r$ depends only on $r$. 
However, our construction does not work if $n/r \le m \le \frac{r-1}{r}n$. We leave the determination of $g_r(n, m)$ in this range of $m$ as an open problem.

\begin{problem}\label{prob:linear_range}
Determine $g_r(n, m)$ for $n/r \le m \le \frac{r-1}{r}n$.
\end{problem}

When the number of edges is in a suitable range we conjecture that the lower bound coming from our construction
$H_{r, s, t}(n)$ (see Section~\ref{sec:constructions}), for some appropriately chosen parameters $s, t$, should be close to the optimal one.
Recall that by $f(n) \ll g(n)$ we mean that $f(n)/g(n) \rightarrow 0$ as $n\rightarrow \infty$.

\begin{conjecture}\label{conj:constant}
Let $m = m(n)$ be a function satisfying $n \ll m \ll n^{{\frac{r+1}{r}}}$. 
Then
\[
    g_r(n, m) = (1+o(1))g_{r}(H_{r, s, t}(n)),
\]
 as $n \rightarrow \infty$ for some choice of $s = s(n)$ and $t = t(n)$.
\end{conjecture}
We remark that an optimization yielded our particular choice of $G_{r, s_1, \ldots, s_{r-1}}$ with $s_1 = 2s$ and 
$s_i = s$ for all $i \neq 1$ (which we simply called $G_{r, s}$). We believe that a similar optimization should yield the right choice of parameters $s = s^*, t = t^*$ to satisfy the conclusion of Conjecture~\ref{conj:constant}.
It seems plausible that the resulting construction $H_{r, s^*, t^*}(n)$ is indeed extremal for $g_r(n,m)$ when $m$ is in the range $n \ll m \ll n^{{\frac{r+1}{r}}}$ given above. 
\paragraph{}
It is natural to consider the largest $k$ such that the Tur\'{a}n subgraph $T_r(k)$ must appear in every $(r+1)$-saturated graph $G$, with $e(G) \geq t_r(n) - m$ edges, where $m \sim Cn^{\frac{r+1}{r}}$. This amounts to the following problem regarding the function $g_r^*(n,m)$, the ``balanced'' analogue of $g_r(n,m)$.

\begin{problem}\label{prob: balanced}
Determine $g^*_r(n,Cn^{\frac{r+1}{r}})$, for each $C \in \mathbb{R}^+$ and sufficiently large $n$.
\end{problem}

Recall that Theorem~\ref{thm:RandomConstruction} shows that $g^*_r(n,Cn^{\frac{r+1}{r}}) \geq \left(1 - \frac{c' \log(Cr)}{C}\right)n$, for $C$ large and fixed and $n \rightarrow \infty$, but we have no non-trivial upper bounds for $g^{*}_r(n,Cn^{\frac{r+1}{r}})$, when $C$ is large. 

\section{Acknowledgements} 
Part of this research was conducted at the University of Cambridge. We are grateful to the Combinatorics Group at Cambridge and for the hospitality of Trinity College, Cambridge. We should also like to thank B\'{e}la Bollob\'{a}s, Matthew Jenssen, and the anonymous referees for many helpful comments and suggestions that greatly improved the presentation of this paper.

\bibliographystyle{siam}
\bibliography{Turan_refs}

\begin{thebibliography}{10}

\bibitem{afgs}
{\sc K.~Amin, J.~Faudree, R.~J. Gould, and E.~Sidorowicz}, {\em {O}n the
  non-$(p-1)$-partite {$K_p$}-free graphs}, Discuss. Math. Graph Theory, 33
  (2013), pp.~9--23.

\bibitem{aes}
{\sc B.~Andr{\'a}sfai, P.~Erd\H{o}s, and V.~S{\'o}s}, {\em On the connection
  between chromatic number, maximal clique and minimal degree of a graph},
  Discrete Math., 8 (1974), pp.~205--218.

\bibitem{bol}
{\sc B.~Bollob{\'a}s}, {\em Modern Graph Theory}, vol.~184 of Graduate Texts in
  Mathematics, Springer-Verlag, New York, 1998.

\bibitem{b}
{\sc A.~Brouwer}, {\em Some lotto numbers from an extension of {T}ur{\'a}n's
  theorem}, Afdeling Zuivere Wiskunde [Department of Pure Mathematics], 152
  (1981).

\bibitem{erd2}
{\sc P.~Erd\H{o}s}, {\em Some recent results on extremal problems in graph
  theory ({R}esults)}, in Theory of Graphs, P.~Rosenstiehl, ed., Gordon and
  Breach, New York, Dunod, Paris, 1967, pp.~118--123.

\bibitem{erd1}
\leavevmode\vrule height 2pt depth -1.6pt width 23pt, {\em On some new
  inequalities concerning extremal properties of graphs}, in Theory of Graphs,
  P.~Erd\H{o}s and G.~Katona, eds., Academic Press, New York, 1968, pp.~77--81.

\bibitem{ht}
{\sc D.~Hanson and B.~Toft}, {\em $k$-saturated graphs of chromatic number at
  least $k$}, Ars. Combin., 31 (1991), pp.~159--164.

\bibitem{kp}
{\sc M.~Kang and O.~Pikhurko}, {\em Maximum {$K_{r+1}$}-free graphs which are
  not $r$-partite}, Mat. Stud., 24 (2005), pp.~12--20.

\bibitem{Nikiforov}
{\sc V.~Nikiforov}, {\em Some new results in extremal graph theory}, in Surveys
  in combinatorics 2011, London Math. Soc. Lecture Note Ser., 392, Cambridge
  Univ. Press, Cambridge, 2011, pp.~141--181.

\bibitem{Nikiforov_Rousseau}
{\sc V.~Nikiforov and C.~{\relax C}. Rousseau}, {\em Large generalized books
  are $p$-good}, J. Combinatorial Theory Ser. B, 92 (2004), pp.~85--97.

\bibitem{s}
{\sc M.~Simonovits}, {\em A method for solving extremal problems in graph
  theory, stability problems}, in Theory of Graphs (Proc. Colloq. Tihany,
  1966), Academic Press, New York, 1968, pp.~279--319.

\bibitem{t}
{\sc P.~Tur{\'a}n}, {\em On an extremal problem in graph theory(in
  {H}ungarian)}, Math. Fiz. Lapok, 48 (1941), pp.~436--452.

\bibitem{tu}
{\sc M.~Tyomkyn and A.~J. Uzzell}, {\em Strong {T}ur{\'a}n stability},
  Electronic Journal of Combinatorics, 22 (2015).

\bibitem{TU_euro}
{\sc M.~Tyomkyn and A.~J. Uzzell}, {\em Strong {T}ur{\'{a}}n stability},
  Electronic Notes in Discrete Mathematics, 49 (2015), pp.~433--440.

\end{thebibliography}

\end{document}